\documentclass{daj}

\usepackage{amsmath,amsthm,amssymb}

\theoremstyle{plain}
\newtheorem{theorem}{Theorem}

\newtheorem{lemma}[theorem]{Lemma}
\newtheorem{corollary}[theorem]{Corollary}

\theoremstyle{definition}
\newtheorem{remark}[theorem]{Remark}

\newcommand\eps{\epsilon}

\newcommand\Z{\mathbb{Z}}

\newcommand\C{\mathbb{C}}

\newcommand\N{\mathbb{N}}

\newcommand\F{\mathbb{F}}
\newcommand\Q{\mathbb{Q}}

\def\p{\mathfrak{p}}

\dajAUTHORdetails{%
  title = {On the Lehmer Conjecture and Counting in Finite Fields}, 
  author = {Emmanuel Breuillard and P\'eter P. Varj\'u},
  plaintextauthor = {Emmanuel Breuillard and Peter P. Varju},
  keywords = {finite fields, Lehmer conjecture, Mahler measure},
}   

\dajEDITORdetails{%
   year={2019},
   number={5},
   received={1 November 2018},   
   published={21 May 2019},  
   doi={10.19086/da.8306},       
}   

\begin{document}

\begin{frontmatter}[classification=text]


\author[EB]{Emmanuel Breuillard\thanks{Supported by ERC Grant no. 617129 `GeTeMo'}}
\author[PV]{P\'eter P. Varj\'u\thanks{Supported by the Royal Society}}

\begin{abstract}
We give a reformulation of the Lehmer conjecture about algebraic integers in terms of a simple counting problem modulo $p$.
\end{abstract}
\end{frontmatter}





\vspace{2em}

Let $\mathcal{S}_d$ be the set of all polynomials of degree at most $d$ with coefficients in $\{0,1\}$.
Recall that if $\alpha$ is an algebraic number, with minimal polynomial $\pi_\alpha(X)=a_0\prod_1^n(X-\alpha_i)$ in $\Z[X]$, then the Mahler measure of $\alpha$ is the quantity
$$M(\alpha) = |a_0| \prod_1^n \max\{1,|\alpha_i|\}.$$

In \cite{breuillard-varju-entropy} we showed that the Mahler measure is related to the growth rate of the cardinality of the set 
$$\mathcal{S}_d(\alpha):=\{P(\alpha) ; P \in \mathcal{S}_d\}$$
as follows.

\begin{theorem}[{\cite[Thm 5, Lem 16]{breuillard-varju-entropy}}] \label{thm:bv} 
 There is a numerical constant $c>0$ such that for every algebraic number $\alpha$, 
$$\min\{2,M(\alpha)\}^c \leq \lim_{d \to+\infty} |\mathcal{S}_d(\alpha)|^{\frac{1}{d}} \leq \min\{2,M(\alpha)\}.$$
\end{theorem}

One can take $c=0.44$. The limit above always exists by sub-multiplicativity of $d \mapsto |\mathcal{S}_d(\alpha)|$. The celebrated Lehmer conjecture posits the existence of a numerical constant $c_0>0$ such that every algebraic number $\alpha$ that is not a root of unity must have
$$M(\alpha)  >1 + c_0.$$ 

In particular the following reformulation of the Lehmer conjecture follows immediately from Theorem~\ref{thm:bv}.

\begin{corollary} The following are equivalent.
\begin{enumerate}
\item  There exists $c_1>0$ such that 
$$\lim_{d \to+\infty} |\mathcal{S}_d(\alpha)|^{\frac{1}{d}} > 1+c_1$$
for every complex number $\alpha$ that is not a root of unity.
\item The Lehmer conjecture is true.
\end{enumerate}
\end{corollary} 

Note that if $\alpha$ is transcendental, or simply not a root of a polynomial of the form $P-Q$, $P,Q$ in $\mathcal{S}_d$ for some $d$, then $ |\mathcal{S}_d(\alpha)|=2^{d+1}$ for all $d$. Note further that if $\alpha$ is a root of unity, then $|\mathcal{S}_d(\alpha)|$ grows at most polynomially in $d$. 

 The set $\mathcal{S}_d(\alpha)$ is the support of the random process $\sum_0^d \epsilon_i \alpha^i$ where the $\epsilon_i$ are Bernoulli random variables equal to $0$ or $1$ with probability $\frac{1}{2}$. The proof of the lower bound in Theorem \ref{thm:bv} consists in establishing a lower bound on the entropy of the above process. 

The purpose of this short note is to give a mod-$p$ version of the above equivalence, showing that Lehmer's conjecture is equivalent to an easily formulated assertion about finite fields.

We first propose the following reformulation. We write $\log$ for the base $2$ logarithm. Given $C\geq 1$, we say that a prime $p$ is $C$-wild if there is $x \in \F_p^{\times}$ of multiplicative order at least $(\log p)^2$ such that not every element of $\F_p$ is a sum of at most $(\log p)^C$ elements from the geometric progression $H:=\{x,x^2,\ldots,x^{[C \log p]}\}$, where $[y]$ denotes the integer part of $y$. 
(Here we allow choosing the elements of $H$ multiple times.)

\begin{theorem}\label{second} The following are equivalent.
\begin{enumerate}
\item[$(a)$] The Lehmer conjecture is true. 
\item[$(b)$] For some  $C> 1$,  almost no prime is $C$-wild. That is, as $X \to +\infty$
$$|\{p \leq X \textnormal{ and } p \textnormal{ is } C\textnormal{-wild} \}| = o(|\{p \leq X\}|).$$
\end{enumerate}
\end{theorem}

The implication $(b)\Rightarrow(a)$ is the easier one, so this result should not be regarded as making any progress towards a
positive answer to the Lehmer conjecture.
The main point of this note is to show the converse implication.
A crucial ingredient in the proof is the lower bound for $|\mathcal{S}_d(\alpha)|$ in Theorem \ref{thm:bv}.

As it turns out the method of proof actually yields a  more precise result: the validity of Lehmer's conjecture implies a stronger statement than $(b)$, and vice versa a weaker version of $(b)$ already implies Lehmer's conjecture. We spell this out in Theorem \ref{main} below and to this aim we first introduce some notation.

Let $\delta,\epsilon>0$. We will say that a prime number $p$ is \emph{$\delta$-bad} if there is a non-zero residue class $x$ modulo $p$ with multiplicative order at least $\log p \log \log \log p$ such that $|\mathcal{S}_d(x)| \leq p^\delta$ for all $d \leq \log p$, where 
$$\mathcal{S}_d(x):=\{P(x) ; P \in \mathcal{S}_d\}.$$

For $K \in \N$, we let $\mathcal{S}^K_d$ be the family of polynomials of degree at most $d$ with integer coefficients in the interval $[-K,K]$. Clearly  \begin{equation}\label{sumset}\mathcal{S}_d \subset  \mathcal{S}^K_d \subset \mathcal{S}_d \pm \ldots \pm  \mathcal{S}_d \subset  \mathcal{S}^{2K}_d,\end{equation} where there are $2K$ summands in the sumset.

We will say that $p$ is \emph{$(\delta,\epsilon)$-very bad} if there is a non-zero residue class $x$ modulo $p$ with multiplicative order at least $\sqrt{\frac{p}{\log p}}$ such that $|\mathcal{S}^{K}_d(x)| \leq p^\delta$ for all $d \leq \frac{1}{\delta}\log p$ and all $K\leq 2^{\epsilon d}$.

We can now state our second reformulation.

\begin{theorem}\label{main} The following are equivalent.
\begin{enumerate}
\item[$(a)$] The Lehmer conjecture is true. 
\item[$(b)$] There is $\delta>0$ such that for all $\epsilon>0$ as $X\to +\infty$,
\begin{equation}
|\{p \leq X \textnormal{ and } p \textnormal{ is } (\delta,\epsilon)\textnormal{-very bad} \}| = o(|\{p \leq X\}|).
\end{equation} 
\item[$(c)$] For every $\varepsilon>0$ there is $\delta>0$ such that as $X\to +\infty$,
\begin{equation}
|\{p \leq X ; p \textnormal{ is } \delta\textnormal{-bad} \}| \ll_\varepsilon X^\varepsilon.
\end{equation} 
\end{enumerate}
\end{theorem}

We have used the Vinogradov notation $\ll_\epsilon$ to mean that the inequality holds up to a multiplicative constant depending only on $\epsilon$.

It is clear that $(c)$ implies $(b)$ in the above theorem, because for $\delta,\epsilon \in(0,\frac{1}{2})$ every large enough $(\delta,\epsilon)$-very bad prime must be $\delta$-bad.

The relevance of the Lehmer conjecture to the study of the family of polynomials $\mathcal{S}_d$ is well-known and appears for example in Konyagin's work \cite{konyagin1, konyagin2}. 

As we mentioned above, the growth of $|\mathcal{S}_d(\alpha)|$ is intimately related to the behavior of the random process $\sum_0^d \epsilon_i \alpha^i$ where the $\epsilon_i$ are Bernoulli random variables equal to $0$ or $1$ with probability $\frac{1}{2}$.
These processes have been studied by Chung, Diaconis and Graham \cite{CDG} and  Hildebrand \cite{hildebrand}.
Very recently they have been used to study irreducibility of random polynomials in \cite{breuillard-varju-random-poly}.

In the remainder of this paper we prove the above theorems. We first handle Theorem \ref{main} and then deduce Theorem \ref{second}.

\section{Proof that $(b)$ implies $(a)$ in Theorem \ref{main}}

We recall the following fairly well-known lemma (see  Lemma 16 in \cite{breuillard-varju-entropy} and \S 8 in \cite{breuillard-solvable}).

\begin{lemma}\label{K} Let $\alpha$ be an algebraic number. Then there is a constant $C_\alpha>0$ such that for every $d,K \geq 1$,
$$|\mathcal{S}_d^{K}(\alpha)| \leq C_\alpha (Kd)^{C_\alpha} M(\alpha)^{d}.$$
\end{lemma}

If Lehmer's conjecture fails, then for each $\delta>0$ there is an algebraic unit $\alpha$ such that $M(\alpha)<1+\delta^2/2$. By the lemma above we conclude that if we choose $\epsilon$ in the interval $(0,C_\alpha^{-1} \delta^2 /4)$, then for all large enough $d$ we have:
$$|\mathcal{S}^{2^{\epsilon d}}_d(\alpha)| \leq 2^{\delta^2 d}.$$

Let $F$ be the Galois closure of $\Q(\alpha)$. Let $p$ be a prime number which splits completely in $F$ and let $\p$ be a prime ideal in $F$ above $p$. Denoting by $\bar{\alpha}$ the residue class of $\alpha$ modulo $\p$ we must have:
$$|\mathcal{S}^{2^{\epsilon d}}_d(\bar{\alpha})| \leq 2^{\delta^2 d},$$ for all sufficiently large $d \in \N$.  In particular if $d \leq \frac{1}{\delta} \log p$, we get 
$$|\mathcal{S}^{2^{\epsilon d}}_d(\bar{\alpha})| \leq p^{\delta}.$$
In particular if the multiplicative order of $\bar{\alpha}$ is at least $\sqrt{\frac{p}{\log p}}$, then $p$ is $(\delta,\epsilon)$-very bad. 

On the other hand, the Frobenius-Chebotarev density theorem tells us that there is a positive proportion of primes $p$ which split completely in $F$. Therefore in order to complete the proof, it is enough to show that there is only a vanishing proportion of primes $p$ for which  $\bar{\alpha}$ has multiplicative order  at most  $\sqrt{\frac{p}{\log p}}$.

\begin{lemma} Let $F|\Q$ be a finite Galois number field. Let $\mathcal{O}_F$ be its ring of integers and let $\alpha\in \mathcal{O}_F \setminus\{0\}$. Let $\mathcal{P}_\alpha$ be the set of primes $p \in \N$ such that there is a prime ideal $\p$ in $F$ above $p$ such that $|\mathcal{O}_F/\p|=p$ and $\alpha \mod \p$ is non-zero and of multiplicative order in $\F_p^\times$ at most  $\sqrt{\frac{p}{\log p}}$. Then as $X\to +\infty$, $$|\{p\leq X ; p \in \mathcal{P}_\alpha\}| = o(|\{p \leq X\}|).$$
\end{lemma}

\begin{proof}Note that if $\alpha^n-1 \in \p$, then $p=N_{F|\Q}(\p)$ divides $N_{F|\Q}(\alpha^n-1)$. In particular 
$$\prod_{p\in \mathcal{P}_\alpha ;  \sqrt{X} \leq p \leq X} p \quad\textnormal{ divides }\quad \prod_{n \leq (X/\log X)^{1/2} }N_{F|\Q}(\alpha^n-1).$$
 On the other hand,
$$|N_{F|\Q}(\alpha^n-1) |\leq \prod_{\sigma \in Gal(F|\Q)} (1+|\sigma(\alpha)|^n) \leq \prod_{\sigma \in Gal(F|\Q)} 2 \max\{1,|\sigma(\alpha)|\}^n \leq 2^{[F:\Q]} M(\alpha)^n$$
We consider this inequality for $n=1,\ldots,(X/\log X)^{1/2}$ and get
$$ |\mathcal{P}_\alpha \cap [\sqrt{X},X]|  \log X  \ll (\frac{X}{\log X})^{1/2} [F:\Q] + \frac{X}{\log X} \log M(\alpha).$$
The conclusion follows immediately given that $|\{p\le X\}|\gg X/\log X$.
\end{proof}

\begin{remark} Clearly we can replace  $(p/\log p)^{1/2}$  by  $p^{1/2}\eps(p)$ for any function $\eps(p)$ tending to $0$ as $p \to +\infty$ with the same proof.
Erd\H{o}s and Murty improve this even further to $p\eps(p)$ in \cite{erdos-murty} in the special case $F=\Q$ assuming the Riemann hypothesis for certain Dedekind zeta functions.
\end{remark}

\section{Proof that $(a)$ implies $(c)$ in Theorem \ref{main}}
The proof of this implication uses the harder part (the lower bound) in Theorem~\ref{thm:bv}.

Let $\mathcal{P}_d$ be the set of polynomials with degree at most $d$ and coefficients in $\{-1,0,1\}$. Note that $\mathcal{P}_d= \mathcal{S}_d - \mathcal{S}_d$. Let $\mathcal{I}_d$ be the set of monic $\Q$-irreducible polynomials dividing at least one non-zero polynomial in $\mathcal{P}_d$. We will say that a prime $p$ is \emph{$d$-exceptional} if it divides the resultant $Res(D_1,D_2)$ of some pair of distinct irreducible polynomials $D_1,D_2$  in $\mathcal{I}_d$. 

Recall that given a field $F$  every field element $\alpha \in F$ determines a partition $\pi_{\alpha,d}$ of $\mathcal{S}_d$ made of the preimages of the map $P \mapsto P(\alpha)$ from $\mathcal{S}_d$ to $F$. Clearly $|\mathcal{S}_d(\alpha)|$ is the number of parts of $\pi_{\alpha,d}$.

\bigskip

\noindent \emph{Claim 1}. If $p$ is not $d$-exceptional, then for every $\alpha \in \F_p$ there is at most one $D \in \mathcal{I}_d$ such that $D(\alpha)=0$ and moreover $\pi_{\alpha,d} = \pi_{\beta,d}$ if $\beta \in \Q$ is a root of $D$.

\bigskip

\begin{proof} If $p$ is not $d$-exceptional, then the reductions mod $p$ of the polynomials in $\mathcal{I}_d$ are pairwise relatively prime. In particular they cannot have a common root in $\F_p$. So each $\alpha \in \F_p$ can be the root of at most one $D \in \mathcal{I}_d$. If there is such a $D$, it is unique, and for any $P,Q \in \mathcal{S}_d$  we have the equivalences : $P(\alpha)=Q(\alpha)$ if and only if $D$ divides $P-Q$ and if and only if $P(\beta)=Q(\beta)$, where $\beta$ is chosen to be a complex root of $D$. This ends the proof. 
\end{proof}

Note that if $\pi_{\alpha,d} = \pi_{\beta,d}$, then $|\mathcal{S}_n(\alpha)| = |\mathcal{S}_n(\beta)|$ for all $n\leq d$. Similarly for any $n\leq d$ we have  $\alpha^n=1$ in $\F_p$ if and only if $\beta^n=1$ in $\overline{\Q}$.
Here we used that the $n$'th cyclotomic polynomial is in $\mathcal{I}_d$ for $n\leq d$.

\bigskip

\noindent \emph{Claim 2}. For $d$ large enough, there are at most $10^d$ $d$-exceptional primes.

\bigskip

\begin{proof} If $D_1,D_2$ are two polynomials in $\mathcal{I}_d$, their resultant can be bounded above by Hadamard's inequality:
$$|Res(D_1,D_2)| \leq \|D_1\|_2^{\deg D_2} \|D_2\|_2^{\deg D_1},$$   where $\|P\|_2$ denotes the $\ell^2$-norm of the coefficients of a polynomial $P$. Recall that for every $P \in \C[X]$ we have $\|P\|_1 \leq 2^{\deg P} M(P)$. This is easily seen by expressing the coefficients of $P$ as sums of products of roots of $P$. Recall also that $M(P) \leq \|P\|_1$ as can be seen using the Jensen formula for $M(P)$.
Here $\|P\|_1$ denotes the $\ell^1$-norm of the coefficients of a polynomial $P$.
However if $D \in \mathcal{I}_d$, then $D$ divides some $P \in \mathcal{P}_d$ and hence $M(D) \leq M(P) \leq \|P\|_1 \leq d+1$, and thus $$\|D\|_2 \leq \|D\|_1 \leq 2^d(d+1).$$
It follows that
$$|Res(D_1,D_2)| \leq 2^{2d^2}(d+1)^{2d} \leq 2^{4d^2}.$$
The number of distinct prime factors of $|Res(D_1,D_2)|$ is thus at most $4d^2$ when $d$ is large enough. Since there are at most $d$ irreducible factors of $P$ for any $P \in \mathcal{P}_d$, $|\mathcal{I}_d| \leq d3^{d+1}$. Hence there can only be at most $4d^2|\mathcal{I}_d|^2 \leq 4d^49^{d+1}$ $d$-exceptional primes. 
\end{proof}

We are now ready to conclude the proof of the implication $(c) \Rightarrow (a)$ in Theorem \ref{main}. We assume that Lehmer's conjecture holds, and therefore that there exists $\delta_0>0$ such that $M(\beta)>e^{\delta_0}$ for every algebraic number $\beta$ that is not a root of unity. Let $\delta>0$ and $X \geq 1$. Assume that $\delta < (c \delta_0/2)^2$, where $c$ is the constant from Theorem \ref{thm:bv}. Let $d$ be the integer part of $\sqrt{\delta} \log X$.

\bigskip

\noindent {\emph{Claim 3.}} Let $p\in [\sqrt{X},X]$ be a prime number. If $p$ is $\delta$-bad, then $p$ is $d$-exceptional.

\bigskip

\begin{proof} If $p$ is $\delta$-bad, there is $\alpha \in \F_p\setminus\{0\}$ of multiplicative order at least $\log p\log\log\log p$ such that $|\mathcal{S}_n(\alpha)| \leq p^\delta$ for all $n \leq \log p$. 
In particular $$|\mathcal{S}_{[\frac{1}{2}\log X]}(\alpha)| \leq X^\delta.$$ On the other hand if $p$ is not $d$-exceptional, by Claim 1 above, there is $\beta \in \overline{\Q}$ of degree at most $d$ such that $|\mathcal{S}_n(\alpha)|=|\mathcal{S}_n(\beta)|$ for all $n \leq d$. In particular if say $\delta<\frac{1}{4}$, we have $d \leq \frac{1}{2}\log X$ and thus $$|\mathcal{S}_d(\beta)| \leq X^\delta \leq e^{\sqrt{\delta}(d+1)} \leq e^{2\sqrt{\delta}d}.$$

However according to Theorem \ref{thm:bv}, $\max\{2,M(\beta)\}^c \leq |\mathcal{S}_d(\beta)|^{1/d}$ (this holds for all $d$ by sub-multiplicativity of $d \mapsto |\mathcal{S}_d(\beta)|$). Hence 
$$ M(\beta) \leq e^{\frac{2}{c}\sqrt{\delta}}.$$

Since we assume the validity of the Lehmer conjecture, it follows that $\beta$ must be a root of unity. However $\beta$ is a root of a polynomial $P$ in $\mathcal{P}_d$, hence has degree at most $d$ over $\Q$. Its minimal polynomial must be a cyclotomic polynomial $\Phi_m$ dividing $P$ hence in $\mathcal{I}_d$. But $p$ is not $d$-exceptional, so by Claim 1 we know that $\Phi_m(\alpha)=0$.

 It is well known that there is an absolute numerical constant $C>0$ such that if $x$ is a root of unity in $\overline{\Q}$ and has degree at most $d$ as an algebraic number over $\Q$, then the order of $x$ is at most $C d \log \log d$. Indeed, the degree of a root of unity of order $n$ is given by the Euler totient function, which satisfies the lower estimate:  $\phi(n) \gg n/\log \log n$ (see \cite[Thm 328]{hardy-wright}).

It follows that $\alpha$ has multiplicative order at most $m \leq C_0 d \log \log d$. As this is less than $\log p \log \log \log p$ if $X$ is large enough (and $\delta_0$ chosen small enough), a contradiction. 
\end{proof}

Now combining Claims 2 and 3 we get:
$$|\{p \in [\sqrt{X},X] ; p \textnormal{ is } \delta\textnormal{-bad} \}| \leq 10^d \leq X^{3\sqrt{\delta}},$$ for all $X \geq 1$, hence

$$|\{p \leq X ; p \textnormal{ is } \delta\textnormal{-bad} \}| \leq X^\delta + \sum_{ k \ge 0 ; 2^k \leq 1/\delta} X^{3 \sqrt{\delta}/2^k} \ll_{\delta} X^{3\sqrt{\delta}}.$$
This concludes the proof of Theorem \ref{main}.

\section{Proof of Theorem \ref{second}}

The proof that $(b)$ implies $(a)$ is immediate from the corresponding implication in Theorem \ref{main}. Indeed given  $\epsilon \in (0,\frac{1}{2})$ and $C > 1$ every large enough $(\frac{1}{C},\epsilon)$-very bad prime will be $C$-wild.

In the converse direction, in view of Theorem \ref{main} it is enough to show that assertion $(c)$ of Theorem \ref{main} implies $(b)$ of Theorem \ref{second}. Let $\epsilon=\frac{1}{2}$ and let $\delta$ be given by assertion $(c)$. It is enough to find $C=C(\delta)>1$ such that every $C$-wild prime is $\delta$-bad. We will show the contrapositive. So assume $p$ is not $\delta$-bad. Then for every $x\in \F_p^{\times}$ with multiplicative order at least $(\log p)^2$ we have 
$$|\mathcal{S}_d(x)| \geq p^\delta$$
for some $d \leq \log p$. Now recall  the sum-product theorem over $\F_p$ : 

\begin{theorem}[Sum-product theorem, see 2.58 and  4.10 in \cite{tao-vu}] There is $\epsilon>0$ such that for all primes $p$ and all subsets $A \subset \F_p$ at least one of three possibilities can occur: either $AA+AA+AA=\F_p$, or $|A+A|>|A|^{1+\epsilon}$, or $|AA|>|A|^{1+\epsilon}$. 
\end{theorem}

Note further that $\mathcal{S}_d^{K}(x) + \mathcal{S}_d^{K}(x) \subset \mathcal{S}_d^{2K}(x)$, while  $\mathcal{S}_d^{K}(x) . \mathcal{S}_d^{K}(x) \subset \mathcal{S}_{2d}^{dK^2}(x)$. In particular, starting from  $\mathcal{S}_d(x)$ and applying at most $n$ times either a sumset or a product set, we obtain a subset of $\mathcal{S}_{2^nd}^{d^{2^n}}(x)$.

Applying the sum-product theorem, we see that after some $n=n(\delta)$ steps $\mathcal{S}_{2^n[\log p]}^{[\log p]^{2^n}}(x)$ must be all of $\F_p$. Setting $C=2^{n+1}$, we conclude (by $(\ref{sumset})$) that $p$ is not $C$-wild, as desired. This concludes the proof of Theorem \ref{second}.


\section*{Acknowledgments} 

PV wishes to thank Boris Bukh for interesting discussions on subjects related to this note.
In particular, PV first learned about a variant of the implication $(b)\Rightarrow(a)$ in Theorem \ref{second}
from him. 

We are very grateful to the anonymous referee for his or her careful reading of our paper.


\newcommand{\etalchar}[1]{$^{#1}$}
\def\cprime{$'$}

\bibliographystyle{amsplain}


\begin{dajauthors}
\begin{authorinfo}[EB]
  Emmanuel Breuillard\\
  DPMMS\\
  University of Cambridge\\
  United Kingdom\\
  emmanuel.breuillard@maths.cam.ac.uk
\end{authorinfo}
\begin{authorinfo}[PV]
  P\'eter P. Varj\'u\\
  DPMMS\\
  University of Cambridge\\
  United Kingdom\\
  pv270@dpmms.cam.ac.uk
\end{authorinfo}
\end{dajauthors}

\end{document}